\documentclass[12pt]{amsart}
\usepackage{amsmath,amsfonts,amsthm,amssymb,mathrsfs, verbatim}
\usepackage{times}
\usepackage{color,ulem}
\theoremstyle{plain}
\newtheorem{theorem}{Theorem}[section]

\newtheorem{lem}[theorem]{Lemma}
\newtheorem{prop}[theorem]{Proposition}
\newtheorem{cor}[theorem]{Corollary}
\theoremstyle{definition}

\newtheorem{defn}[theorem]{Definition}
\newtheorem{Ex}[theorem]{Example}
\theoremstyle{remark}

\newtheorem*{rmk}{Remark}

\newcommand{\pl}{\partial}

\newcommand{\na}{\nabla}

\newcommand{\lt}{\left}
\newcommand{\rt}{\right}
\newcommand{\rw}{\rightarrow}

\newcommand{\R}{\mathbb{R}}

\newcommand{\mxtr}{\mbox{tr}}
\newcommand{\mxRic}{\mbox{Ric}}

\title{Brendle's inequality on static manifolds}
\author{Xiaodong Wang, Ye-Kai Wang}
\address{Xiaodong Wang\\
Department of Mathematics\\
Michigan State University, U.S.} \email{xwang@math.msu.edu}

\address{Ye-Kai Wang\\
Department of Mathematics\\
Michigan State University, U.S.} \email{wangya@math.msu.edu}

\date{}

\begin{document}
\maketitle
\begin{abstract}
We generalize Brendle's geometric inequality considered in \cite{B} to static manifolds. The inequality bounds the integral of inverse mean curvature of an embedded mean-convex hypersurface by geometric data of the horizon. As a consequence, we obtain a reverse Penrose inequality on static asymptotically locally hyperbolic manifolds in the spirit of Chru\'{s}ciel and Simon \cite{CS}.  
\end{abstract}
\section{Introduction}
We say $(M^n,g,f)$ is a {\it static manifold} if $(M,g)$ is a complete Riemannian manifold, $f$ is a positive function on $M$ ($\{ f=0 \} = \pl M$ if $\pl M \neq\phi$), and satisfies the static equations
\begin{align}
f R_{ij} - D_i D_j f + \Delta f g_{ij} =0, \\
\Delta f = -\Lambda f.
\end{align}
Here $\Lambda$ is called the {\it cosmological constant}. By normalization, we assume $\Lambda = \epsilon n$ with $\epsilon \in \{ 1,0,-1\}$. 

Interests in static manifolds are motivated by  general relativity. The (Lorentzian) metric
\[ \bar{g} = -f^2 dt^2 + g \]
solves the (Lorentzian) vacuum Einstein equations with cosmological constant $\Lambda$ on $(-\infty,\infty) \times M$. An extensive study has been done on 3-dimensional static manifolds with zero cosmological constant and asymptotically flat at infinity (see \cite[Assumption 1, page 148]{BM}). The case of black hole boundary is completely classified: $(M,g,f)$ is isometric to the exterior Schwarzschild solution \cite{I, R, BM, Ch}.   
Recently, Cederbaum and Galloway solved the case of photon sphere boundary \cite{CG}. On the other hand, the static $n$-body problem is far from settled, see \cite{BS}. A better understanding of the behavior of the static potential \cite{GM, MT} would be crucial. 

In \cite{B}, Brendle considers a class of warped product manifolds $(M,g)$ that includes the (anti) de Sitter-Schwarzschild manifold. Here $M = [0,\bar{r}) \times N$ and 
\[ g = dr \otimes dr + h(r)^2 g_N. \]
The assumptions Brendle imposes on the warping factor essentially require that $f(r)=h'(r)$ satisfies
\[ f R_{ij} - D_i D_j f + \Delta f g_{ij} \geq 0 \]
and $f(0)=0$. Brendle proves a geometric inequality
\begin{theorem}\cite[Theorem 3.5, 3.11]{B}
Let $\Sigma$ be an embedded hypersurface with positive mean curvature $H$. 
\begin{enumerate}
\item If $\Sigma = \pl \Omega$, then
\begin{align}\label{null-homologous}
(n-1) \int_\Sigma \frac{f}{H} d\mu \ge n \int_\Omega f dvol
\end{align}
\item If $\Sigma$ is homologous to $\{ 0\} \times N$ so that $\pl \Omega = \Sigma \cup \{ 0\} \times N$, then
\begin{align}
(n-1) \int_\Sigma \frac{f}{H} d\mu \ge n\int_\Omega f dvol + h(0)^n vol(N,g_N)
\end{align}
\end{enumerate}
In both cases, the equality holds if and only if $\Sigma$ is umbilical.
\end{theorem}

Applications of this inequality include the classification of embedded constant mean curvature hypersurfaces \cite{B} and a sharp Minkowski inequality \cite{BHW} on (anti) de Sitter-Schwarzschild manifolds. 

M. Eichmair observes that the first inequality holds on static manifolds (see the footnote on page 257 of \cite{B}). The main result in this note is a generalization of the second inequality to the static manifolds. 
\begin{theorem}\label{Heintze-Karcher on static}
Let $N$ be a connected component of $\pl M$ and $\Sigma$ be an embedded hypersurface that is homologous to $N$, so that $\pl \Omega = \Sigma \cup N$. Assume that the scalar curvature of $N$ satisfies
\[ R^N > \epsilon n (n-1) \]
and $\Sigma$ has positive mean curvature with respect to the outward normal of $\Omega$. Then
\begin{align}\label{main}
(n-1) \int_\Sigma \frac{f}{H} d\mu \ge n \int_\Omega f \,d \mbox{vol} + \dfrac{(n-1) \kappa}{\max_N \lt( \frac{R^N - \epsilon n (n-1)}{2} \rt)} \cdot \mbox{vol}(N).
\end{align}
Here $\kappa = |df|_g$ on $N$. Moreover, if the equality holds, then $\Sigma$ is umbilic.
\end{theorem}

Starting from section 3, we study the implications of (\ref{main}) on $(M,g,f)$ with various boundary conditions. When the cosmological constant is negative, a major open problem in general relativity is to classify all asymptotically hyperbolic (see Section 4 for the definition) static manifolds. The only known case is when $\pl M=\phi$, Boucher-Gibbons-Horowitz \cite{BGH}, J. Qing \cite{Q} and the first author \cite{W} showed that the hyperbolic space is the only solution.

We observe that the integral
\[ (n-1) \int_\Sigma \frac{f}{H} d\mu - \int_\Sigma \frac{\pl f}{\pl \nu} d\mu \]
approaches negative ADM mass when $\Sigma$ approaches infinity. Hence, (\ref{main}) can be regarded as a reverse Penrose inequality. 
\begin{cor}\label{reverse Penrose}
Let $(M,g,f)$ be an asymptotically locally hyperbolic static manifold with a connected boundary $N$. Assume (\ref{assumption, R^N}) holds on $N$. Then we have an upper bound for the mass of $(M,g)$:
\begin{align}
m \le \frac{\kappa}{(n-2)\,\omega_{n-1}} \lt( 1 - \frac{n-1}{\max_N \lt( \frac{R^N + n(n-1)}{2} \rt)}\rt) \, \mbox{vol}(N), 
\end{align}
where $\omega_{n-1}$ is the volume of the $(n-1)$-dimensional unit sphere.
\end{cor}

Chru\'{s}ciel and Simon \cite{CS} suggest that such inequality can be combined with a Penrose-type inequality to obtain uniqueness results of static manifolds. We refer to \cite{LN} for a realization of this program when the mass aspect function is nonpositive.

Let us turn to the case of positive cosmological constant. Besides the hemisphere with standard round metric and $f = \cos\theta$, we have the family of de Sitter-Schwarzschild manifolds as models. It is a warped product with $M = [s_1,s_2] \times S^{n-1}$ and 
\[ g = \frac{1}{f^2} ds^2 + s^2 g_0 \]
where $f = \sqrt{1-s^2 - m s^{2-n}}$ and $s_1,s_2$ are two positive roots of $f$. The equator of the hemisphere and $\{ s=s_2\}$ in de Sitter-Schwarzschild manifold are called cosmological horizons in physics literature. The presence of them makes the classification problem qualitatively different from that of zero or negative cosmological constant. We recommend the introduction of \cite{A} for the current status in dimension 3. Here we only mention the main result in \cite{A}: 
\begin{theorem}\cite[Theorem D]{A}
Let $(M^3,g,f)$ be a compact simply connected static manifold with connected boundary and scalar curvature 6. If
\[ |\pl M| \ge \frac{4\pi}{3}, \]
then $(M^3,g,f)$ is equivalent to the standard hemisphere.
\end{theorem}   

Motivating by de Sitter-Schwarzschild manifolds, we relate the geometric quantity of two horizons in a static manifold. \begin{cor}\label{compare 2 horizons} Let $(M,g,f)$ be a static manifold with positive cosmological constant. Suppose that $M$ is diffeomorphic to $N \times [s_1, s_2]$. Let $N_1 = N \times \{ s_1\}$ and $N_2 = N \times \{ s_2\}$ be the two connected components of $\pl M$. Assume that \[ R^{N_1} > n(n-1)
\quad\mbox{ and} \quad R^{N_2}< n(n-1). \] Then we have
\begin{align}
\kappa_2 \int_{N_2} \lt( \dfrac{2(n-1)}{n(n-1) - R^{N_2}} -1 \rt) d\mu  \ge \kappa_1 \lt( \frac{2(n-1)}{\max_{N_1} \lt( R^{N_1} - n(n-1) \rt)} +1 \rt) \cdot vol(N_1).
\end{align}
Here $\kappa_1$ and $\kappa_2$ are the surface gravity $|Df|_g$ on $N_1$ and $N_2$ respectively.
\end{cor}     

The paper is organized as follows. After reviewing basic facts on static manifolds (Section 2), we prove Theorem \ref{Heintze-Karcher on static} in Section 3. As applications to Theorem \ref{Heintze-Karcher on static}, we prove Corollary \ref{reverse Penrose} in Section 4 and Corollary \ref{compare 2 horizons} in Section 5 respectively. 

{\bf Acknowledgements.} The first author is partially supported by Simons Foundation Collaboration Grant for Mathematicians \#312820. The second author would like to thank Professor Mu-Tao Wang for his constant encouragement.

\section{Geometry near the horizon}
We assume that $(M,g,f)$ is an $n$-dimensional static manifold with \underline{connected} boundary $N$.  Let $D$ be the Levi-Civita connection of $g$. The static equations now read
\begin{align*}
f(R_{\alpha\beta} - n \epsilon g_{\alpha\beta}) = D_\alpha D_\beta f \\
\Delta_g f = -n\epsilon f
\end{align*}
where $\epsilon \in \{ 1,0,-1 \}$. Static equations imply that $R =  \epsilon n(n-1)$ and $|Df|$ is a positive constant $\kappa$ on $N$. Moreover, $N$ is totally geodesic.
\begin{defn}
The constant $\kappa$ is called the {\it surface gravity} of the horizon $N$.
\end{defn}

 Let $r$ be the distance function from $N$. We fix a tubular neighborhood $U_N$ of $N$ that is diffeomorphic to $N \times [0,r_0)$. Fixing a coordinate system $\{ x^i \}$ on $N$, we henceforth work on $N \times [0,r_0)$ with coordinates $\{ r,x^i\}$ instead of $U_N$. The expansion of the static potential is given by \begin{align}\label{exp, f}
 f = \kappa r + O(r^3).
\end{align}
We can choose $r_0$ sufficiently small such that $f \ge \frac{1}{2} \kappa r$ in $N \times [0, r_0)$.

Let $\tilde{g}$ denote the induced metric of $N$. Since $N$ is totally geodesic, the metric has the expansion
\begin{align}\label{exp, metric}
g = dr^2 + \tilde{g} + r^2 g^{(2)} + O(r^3),
\end{align} where $g^{(2)}$ is a symmetric two-tensor on $N$. Direct computation shows
$R_{irjr}|_{r=0} = -g^{(2)}_{ij}$ and hence 
\[ \mxtr_{\tilde{g}} g^{(2)} = -\mxRic \lt( \frac{\pl}{\pl r}, \frac{\pl}{\pl r} \rt).\] 

We assume
\begin{align}\label{assumption, R^N}
R^N > \epsilon n(n-1).
\end{align}
By Gauss formula,
\[ R^N = R - 2 \mbox{Ric}\lt( \frac{\pl}{\pl r}, \frac{\pl}{\pl r} \rt) + H^2 - |A|^2, \]
and hence (\ref{assumption, R^N}) is equivalent to 
\begin{align}\label{assumption, Ric}
\mxRic \lt( \frac{\pl}{\pl r}, \frac{\pl}{\pl r} \rt) < 0 \quad \mbox{on } N.
\end{align}
By the second variation formula of mean curvature, 
\[ \frac{\pl H}{\pl r} = - |A|^2  - \mbox{Ric}(\frac{\pl}{\pl r}, \frac{\pl}{\pl r}), \]
and the fact that $N$ is totally geodesic, the level sets of $r$ have positive mean curvature with respect to $\frac{\pl}{\pl r}$ for $r$ sufficiently small.
 
The geometry near the horizon admits an approximate conformal Killing vector field.
\begin{prop}\label{approx conf Killing}
There exists a vector field $X$ on $N \times [0,r_0)$ satisfying
\begin{align}\label{conf Killing eq}
\begin{split}
\lt\langle D_{\frac{\pl}{\pl r}} X, \frac{\pl}{\pl r} \rt\rangle &= f + O(r^2), \\
\lt\langle  D_{\frac{\pl}{\pl r}} X, E_i \rt\rangle  + \lt\langle D_{E_i} X, \frac{\pl}{\pl r} \rt\rangle &= O(r^2), \\
\sum_{i=1}^{n-1} \lt\langle D_{E_i} X, E_i \rt\rangle  &=(n-1) f + O(r^2),
\end{split}
\end{align}
where $\{ E_i \}$ are orthonormal frames on $N \times \{ r \}$. Moreover, $X$ has the expansion
\begin{align}
X = \lt( a^{(0)} + \frac{1}{2} \kappa r^2 \rt) \frac{\pl}{\pl r} + r b^i \frac{\pl}{\pl x^i}	+ O(r^3)
\end{align}
where $a^{(0)}$ is a function on $N$ with
\begin{align} \label{min a}
\min_N a^{(0)} \ge \dfrac{(n-1) \kappa}{\max_N \lt( -\mxRic \lt( \frac{\pl}{\pl r}, \frac{\pl}{\pl r}\rt) \rt)},
\end{align}
and $b^i \frac{\pl}{\pl x^i}$ is a vector field on $N$. 
\end{prop}
\begin{proof}
Suppose $X$ is given by the expansion
\begin{align*}
X = \lt( a^{(0)} + r a^{(1)} + r^2 a^{(2)} \rt) \frac{\pl}{\pl r} + \lt( b^{(0)i} + r b^{(1)i} + r^2 b^{(2)i} \rt) \frac{\pl}{\pl x^i} + O(r^3)
\end{align*}
where $a^{(\alpha)}, b^{(\alpha)i}$ are functions and tensors on $N$.

Equations (\ref{conf Killing eq} ) imply the following equations on $N$:
\begin{align*}
a^{(1)} =0, \qquad \frac{\pl a^{(0)}}{\pl x^i}  = b_i^{(1)}, \qquad \tilde{\na}_i b^{(0)i}=0 \\
a^{(2)} = \frac{1}{2} \kappa, \qquad b_i^{(2)} - (g^{(2)})_i^j b^{(2)}_j + \frac{\pl a^{(1)}}{\pl x^i}=0,
\end{align*}
and
\begin{align}
\tilde{\na}_i b^{(1)i} + a^{(0)} \mxtr_{\tilde{g}} g^{(2)} = (n-1) \kappa.
\end{align}
Here we raise and lower indices with respect to $\tilde{g}$.
Setting $b^{(0)i}= b^{(2)i}=0$ solves all but one equation. The last one
\begin{align}\label{Poisson}
- \tilde{\Delta} a^{(0)} - \mxRic \lt( \frac{\pl}{\pl r}, \frac{\pl}{\pl r} \rt) a^{(0)} = (n-1) \kappa
\end{align}
is solvable on $N$ by assumption (\ref{assumption, Ric}) and the maximum principle. Inequality (\ref{min a}) follows from the maximum principle.
\end{proof}

\section{Brendle's geometric inequality on static manifolds}
We closely follow the presentation of \cite[Section 3]{B} and omit the arguments that are identical. The new ingredients are Lemma \ref{improved normal vector estimate} and Lemma \ref{reciprocal inequality}. Lemma \ref{improved normal vector estimate} is an improvement of Lemma 3.6 in \cite{B}. While Proposition 2.3 in \cite{B} is not available in our case because we only have an approximate conformal Killing vector field $X$ near horizon, we replace it by Lemma \ref{reciprocal inequality}.

We recall some definitions from \cite{B}. Let $\Sigma$ be a closed, embedded, orientable, mean-convex hypersurface that is homologous to $N$. Let $\Omega$ be the domain enclosed by $\Sigma$ and $N$. Consider the Fermat metric $\hat{g} = \frac{1}{f^2} g$. For each point $p \in \bar{\Omega}$, we denote by $u(p) = d_{\hat{g}} (p,\Sigma)$ the distance of $p$ from $\Sigma$ with respect to the metric $\hat{g}$. Let $\Phi: \Sigma \times[0,\infty) \rw \bar{\Omega}$ be the normal exponential map with respect to $\hat{g}$. Namely, for each point $x \in \Sigma$, the curve $t \mapsto \Phi(x,t)$ is a geodesic with respect to $\hat{g}$ with
\[ \Phi(x,0)=x, \quad \frac{\pl}{\pl t} \Phi(x,t) \Big|_{t=0} = -f(x) \nu(x). \]

Define
\[ A = \{ (x,t) \in \Sigma \times [0,\infty): u (\Phi(x,t)) = t \}\]
and \[ A^* = \{ (x,t) \in \Sigma \times [0,\infty): (x, t+\delta) \in A \mbox{ for some } \delta>0 \}. \]

\begin{lem}\label{improved normal vector estimate}
For any $0 < \sigma <1$, there exist a number $\tau_1>0$ with the following property: if $p$ is a point in $\{ u \ge \tau_1 \}$ and $\alpha$ is a unit-speed geodesic with respect to $\hat{g}$ such that $\alpha(0)=p$ and $\alpha( u(p)) \in \Sigma,$ then $| \alpha'(0)| = f(p)$ and
\[ \lt\langle \frac{\pl}{\pl r}, \alpha'(0) \rt\rangle  \ge f(p) - c f(p)^{3- \sigma}, \]
where $c$ is a constant independent of $u(p)$.
\end{lem}
\begin{proof}
We follow Brendle's idea. Since $D^2 f = O(r)$ and $Df = \kappa \frac{\pl}{\pl r} + O(r^2)$, we can find a small number $r_1$ such that 
\begin{align}\label{Hessian inequality}
-f D^2 f + |Df|^2 g \ge \kappa^2 (1-f) g
\end{align}
and 
\begin{align}\label{simple inequality}
\frac{d}{dt} f(\alpha(t)) = \langle Df(\alpha(t)), \alpha'(t) \rangle \le |Df (\alpha(t))|  \cdot |\alpha'(t)| \le \frac{\kappa}{1-\sigma} f(\alpha(t))
\end{align}
on the set $N \times (0,r_1]$. By (\ref{Hessian inequality}), the Hessian of the function $\frac{1}{f}$ with respect to $\bar{g}$ satisfies
\[ \hat{D}^2 \lt( \frac{1}{f} \rt) \ge \kappa^2 \lt( \frac{1}{f} - 1 \rt) \hat{g} \]
on the set $N \times (0,r_1].$ Let $T$ be a large number such that
\[ \dfrac{e^{\kappa t} + e^{-\kappa t} + 2}{e^{\kappa t} - e^{-\kappa t} + 2} \ge 1 + 2 e^{-2\kappa t} \mbox{ and } - \frac{4}{e^{\kappa t} - 3 e^{-\kappa t} + 2} \ge -5 e^{-\kappa t} \]
for $t \ge T$. Let \[ \tau_1 = \max_{N \times \{ r_1\}} u + T . \] Consider a point $p \in \{ u \ge \tau_1\}$ and a unit-speed geodesic $\alpha$ with respect to $\hat{g}$ such that $\alpha(0)=p$ and $\alpha(u(p)) \in \Sigma$.  We now define $t_1 = \inf \{ t \in [0,u(p)] : \alpha(t) \notin N \times (0,r_1].$ Clearly $t_1 \ge T$.  
Moreover, we have
\[ \frac{d^2}{dt^2} \lt( \frac{1}{f(\alpha(t))} \rt) \ge \kappa^2 \lt( \frac{1}{f(\alpha(t))} - 1 \rt)\]
for all $t \in [0,t_1]$. Integrating this differential inequality, we obtain
\[ \frac{1}{f(\alpha(t))} \ge \frac{1}{2} \lt( \frac{1}{f(p)} + \frac{\langle Df(p), \alpha'(0) \rangle}{\kappa f(p)^2} \rt) \lt( \cosh (\kappa t) +1 \rt) - \frac{1}{\kappa f(p)^2} \langle Df(p), \alpha'(0) \rangle \lt( \sinh (\kappa t) + 1 \rt)\]
Putting $t=t_1$ and rearranging terms gives
\[ \langle Df(p), \alpha'(0) \rangle \ge \kappa \lt(  \frac{e^{\kappa t_1} + e^{-\kappa t_1} + 2}{e^{\kappa t_1} - 3 e^{-\kappa t_1} + 2} f(p) - \frac{4}{e^{\kappa t_1} - 3e^{-\kappa t_1} + 2} \frac{f(p)^2}{f(\alpha(t_1))} \rt).\]
By our choice of $t_1$, we obtain
\begin{align}\label{intermediate}
 \langle Df(p), \alpha'(0) \rangle \ge \kappa \lt( \lt( 1 + 2 e^{-2\kappa t_1} \rt) f(p) - 5 e^{-\kappa t_1} \frac{f(p)^2}{f(\alpha(t_1))} \rt).
\end{align}

Integrating (\ref{simple inequality}), we get \[ e^{\frac{\kappa}{1-\sigma}  t_1} \ge \dfrac{f(\alpha)(t_1)}{f(p)} \]
and
\[ -e^{-\kappa t_1} \ge \frac{f(p)^{1-\sigma}}{f(\alpha(t_1))^{1-\sigma}}.\]
Inserting this inequality into (\ref{intermediate}), we get
\begin{align*}
\langle Df(p), \alpha'(0) \rangle &\ge \kappa f(p) - \frac{5 \kappa}{f(\alpha(t_1))^{1-\sigma}} f(p)^{3-\sigma}.
\end{align*}
Recalling that $Df = \kappa \frac{\pl}{\pl r} + O(r^2)$, we reach
\[ \lt\langle \frac{\pl}{\pl r}, \alpha'(0) \rt\rangle \ge f(p) - c f(p)^{3-\sigma},\]
as claimed.
\end{proof}

We fix $\sigma$ and $\tau_1$ so that the conclusion of  Lemma \ref{improved normal vector estimate} holds and $f(p) - c f(p)^{3-\sigma} \ge \frac{1}{2} f(p)$. The next two lemmata thus follow from Lemma 3.7 and 3.8 in \cite{B} verbatim.
\begin{lem}
Suppose that $\gamma: [a,b] \rw \{ u \ge \tau_1\}$ is a smooth path satisfying $|\gamma'(s) + f(\gamma(s)) \frac{\pl}{\pl r}|_{\hat{g}} \le \frac{1}{4}$ for all $s \in [a,b].$ Then
\[ u \lt( \gamma(g)\rt) - u \lt( \gamma(a)\rt) \ge \frac{1}{4} (b-a). \]
\end{lem}
\begin{lem}\label{approximation of distance function}
Given any real number $\tau \ge \tau_1 + 2$, there exists a sequence of smooth functions $u_j: \{ \tau-1 < u < \tau+1 \} \rw \R$ with the following properties:
\end{lem}
\begin{enumerate}
\item[(i)] The functions $u_j$ converge smoothly to $u$ away from the cut locus. More precisely, $u_j \rw u$ in $C^\infty_{loc}(W)$, where $W = \Phi(A^* \cap (\Sigma \times (\tau-1,\tau+1))).$
\item[(ii)] For each point $p \in \{ \tau-1 < u < \tau+1\},$ we have $|u_j(p) - u(p)| \le \frac{1}{j^2}$.
\item[(iii)] For all points $p,q \in \{ \tau-1 < u < \tau+1\},$ we have $|u_j(p) - u_j(q)| \le (1 + \frac{1}{j}) d_{\hat{g}}(p,q)$.
\item[(iv)] If $\gamma:[a,b] \rw \{ \tau-1 < u < \tau+1 \}$ is an integral curve of the vector field $-f \frac{\pl}{\pl r},$ then $u_j(\gamma(b)) - u_j(\gamma(a)) \ge \frac{1}{4} (b-a)$.
\item[(v)] We have $\hat{D}^2 u_j \le K(\tau) \hat{g}$ at each point $p \in \{ \tau-1 < u < \tau+1\}$. Here, $K(\tau)$ is a positive constant which may depend on $\tau$, but not on $j$. 
\end{enumerate}

\begin{lem}\label{reciprocal inequality}(cf. \cite[Proposition 2.3]{B}) Given $\tau \ge \tau_1+2$, let $u_j$ be a sequence of smooth functions satisfying properties (i)-(v) in Lemma \ref{approximation of distance function}. Denote the level set $\{ u_j =t_j\}$ by $S_j$, where $t_j \in [\tau+\frac{1}{j^2}, \tau+\frac{1}{j} - \frac{1}{j^2}]$. We have
\begin{align*}
\limsup_{j \rw \infty} \int_{S_j} \frac{H}{f} \langle X,\nu \rangle d\mu \le (n-1) \limsup_{j \rw \infty} \mu (S_j) + O(r^\sigma).
\end{align*}
\end{lem}
\begin{proof}
By Lemma \ref{improved normal vector estimate}, the tangent space of $\Sigma_\tau^*$ at $p$ and $N \times \{ r(p)\}$ differ by $O(r^{1+\sigma})$. We may assume the same holds for $j$ sufficiently large. Given a point $p \in S_j \cap (N \times \{ r\})$ and orthonormal frames $\{ e_i\}_{i=1}^{n-1}$ of $T_p S_j$, we can find orthonormal frame $\{ E_i\}_{i=1}^{n-1}$ of $T_p (N \times r(p))$ such that
\[ e_i = A_i E_i + B_i \frac{\pl}{\pl r}\]
for all $i$ with $A_i = 1 + O(r^{1+\sigma})$ and $B_i = O(r^{1+\sigma})$. 

Let $X^T$ be the tangential projection of $X$ on $S_j$. By Proposition \ref{approx conf Killing}, we have
\begin{align*}
\na_i (X^T)^i &= \langle D_{e_i} X, e_i \rangle - H \langle X, \nu \rangle \\
&= (n-1) f - H \langle X,\nu \rangle + O(r^2).
\end{align*}
Dividing by $f$ and integrating on $S_j$, we get
\begin{align*}
\int_{S_j} \frac{\langle X^T, Df \rangle}{f^2} d\mu = (n-1) \mu(S_j) - \int_{S_j} \frac{H}{f} \langle X,\nu\rangle d\mu + O(r).
\end{align*}
At each point, $X - r b^i \frac{\pl}{\pl x^i}$ is a positive multiple of $Df$ up to $O(r^3)$. Hence, 
\begin{align*}
\int_{S_j} \frac{r b^i \lt\langle\lt( \frac{\pl}{\pl x^i}\rt)^T , Df \rt\rangle}{f^2} d\mu \le (n-1) \mu(S_j) - \int_{S_j} \frac{H}{f} \langle X,\nu \rangle d\mu + O(r)
\end{align*}
By Lemma \ref{improved normal vector estimate}, we have $\lt\langle \lt( \frac{\pl}{\pl x^i}\rt)^T, Df \rt\rangle = O(r^{1+\sigma})$ and hence, 
\[ \int_{S_j} \frac{H}{f} \langle X,\nu \rangle d\mu \le (n-1) \mu(S_j) + O(r^\sigma). \]
From this, the assertion follows.
\end{proof}

The next result is proved identically as in \cite{B}. The only difference is that we use Lemma \ref{reciprocal inequality} instead of \cite[Proposition 2.3]{B} to get the same inequality on $S_j$.
\begin{prop}\cite[Proposition 3.9]{B}
For $\tau \ge \tau_1 + 2$ we have
\begin{align}\label{area}
\mu(\Sigma_\tau^*) \ge \mbox{vol}(N)
\end{align}
and
\begin{align}
\int_{\Sigma_\tau^*} \frac{H}{f} \langle X,\nu \rangle d\mu \le (n-1) \mu(\Sigma_\tau^*) + c(\tau)
\end{align}
for some $c$ satisfying $\lim_{\tau \rw \infty} c(\tau) =0$.
\end{prop}

\begin{cor}\cite[Corollary 3.10]{B}
Let $\lambda \in (0,1)$ be given. Then we have
\begin{align}
(n-1) \int_{\Sigma_\tau^*} \frac{f}{H} d\mu \ge \lambda \lt( \min_N a^{(0)} \rt)  \mbox{vol}(N) + c(\tau) 
\end{align}
for some $c$ satisfying $\lim_{\tau \rw \infty} c(\tau) =0$.
\end{cor}
\begin{proof}
The proof proceeds exactly as in \cite{B} if we replace $h(0)$ by $\min_N a^{(0)}$. 
\end{proof}

\begin{theorem}
Let $\Sigma$ be an embedded hypersurface that is homologous to the boundary $N$, so that $\pl \Omega = \Sigma \cup N$. Assume that $\Sigma$ has positive mean curvature. Then
\begin{align}\label{Heintze-Karcher}
(n-1) \int_\Sigma \frac{f}{H} d\mu \ge n \int_\Omega f \,d \mbox{vol} + \dfrac{(n-1) \kappa}{\max_N \lt( \frac{R^N -\epsilon n(n-1)}{2} \rt)} \cdot \mbox{vol}(N).
\end{align}
Here $R^N$ is the scalar curvature of $N$ with the induced metric $\tilde{g}$. Moreover, if equality holds, then $\Sigma$ is umbilic.
\end{theorem}
\begin{proof}
Taking footnote 1 at the bottom of \cite[page 257]{B} into account, we argue in the same way as in \cite{B} to get
\[ (n-1) \int_\Sigma \frac{f}{H} d\mu \ge n \int_\Omega f\, d\mbox{vol} + \lt( \min_N a^{(0)} \rt) \mbox{vol}(N)\]
By (\ref{min a}) and the Gauss equation, $R^N = \epsilon n(n-1) - 2\, \mxRic(\frac{\pl}{\pl r}, \frac{\pl}{\pl r})$, the result follows.
\end{proof}

With the Brendle type geometric inequality (\ref{Heintze-Karcher}) at hand, we follow \cite{BHW} to get an monotonicity formula for mean-convex hypersurfaces in static manifolds with negative cosmological constant.
\begin{cor}
Suppose $M,N, \Sigma$ satisfy the assumptions of the previous theorem. Let $\Sigma_t$ be the solution of the inverse mean curvature flow with $\Sigma_0 = \Sigma$. Then the quantity
\[ Q(t) = |\Sigma_t|^{-\frac{n-2}{n-1}}\lt( \int_{\Sigma_t} f H d\mu - n(n-1) \int_{\Omega_t} f d\mbox{vol} + \frac{n-1}{n-2} \lt( 2 - \frac{n(n-1)}{\max_N \lt( \frac{R^N + n(n-1)}{2} \rt)}\rt) \kappa \cdot vol(N) \rt)\]
is monotone decreasing under the inverse mean curvature flow. Here $|\Sigma_t|$ denotes the volume of $\Sigma_t$. 
\end{cor}
\begin{proof}
The proof is almost identical to that of \cite[Proposition 19]{BHW}. For the reader's convenience, we include the proof here. 

The evolution of the mean curvature is given by
\[ \frac{\pl}{\pl t} H = - \Delta_\Sigma \lt( \frac{1}{H} \rt) - \frac{1}{H} \lt( |A|^2 + \mbox{Ric}(\nu,\nu) \rt). \]
This implies \[ \frac{\pl}{\pl t}(fH) = - f \Delta_\Sigma \lt( \frac{1}{H} \rt) - \frac{f}{H} \lt( |A|^2 + \mbox{Ric}(\nu,\nu) \rt) + \langle Df, \nu \rangle.\]
Using the identity $\Delta f = \Delta_\Sigma f + D^2 f(\nu,\nu) + H \langle Df, \nu \rangle$ and the static equations, we obtain 
\begin{align*}
\frac{d}{dt} \lt( \int_{\Sigma_t} fH d\mu \rt) &= -\int_{\Sigma_t} \Delta_\Sigma f \cdot \frac{1}{H} d\mu - \int_{\Sigma_t} \frac{f}{H} \lt( |A|^2 + \mbox{Ric}(\nu,\nu) \rt) d\mu + \int_{\Sigma_t} \lt(  \langle Df, \nu \rangle + fH \rt) d\mu\\
&= \int_{\Sigma_t} \lt( 2 \langle Df,\nu \rangle + fH - f \frac{|A|^2}{H} \rt)d\mu \\
&\le 
\int_{\Sigma_t} \lt( 2 \langle Df,\nu \rangle + \frac{n-2}{n-1} fH \rt) d\mu.
\end{align*}
Since $\Delta f = nf$, the divergence theorem implies
\[ \int_{\Sigma_t} \langle Df, \nu \rangle d\mu = n \int_{\Omega_t} f dvol + \kappa \cdot vol(N).\]
Moreover, by (\ref{Heintze-Karcher}) we have
\[ \frac{d}{dt}\lt( -n(n-1) \int_{\Omega_t} f dvol \rt) = -n(n-1) \int_{\Sigma_t} \frac{f}{H} \le -n^2 \int_{\Omega_t} f dvol - \frac{n(n-1)\kappa}{\max_N \lt( \frac{R^N + n(n-1)}{2} \rt)} \cdot vol(N.)\]

If we set 
\[ \bar{Q}(t) = \int_{\Sigma_t} fH d\mu - n(n-1)\int_{\Omega_t} f dvol + \frac{n-1}{n-2} \lt( 2 - \frac{n(n-1)}{\max_N \lt( \frac{R^N + n(n-1)}{2} \rt)}\rt) \kappa \cdot vol(N), \] 
the above facts together imply
\begin{align*}
\frac{d}{dt} \bar{Q}(t) \le \frac{n-2}{n-1} \bar{Q}(t)
\end{align*}
Since $|\Sigma_t| = e^t |\Sigma|$ under the inverse mean curvature flow, 
\[ \frac{d}{dt} Q(t) = \frac{d}{dt} \lt( |\Sigma_t|^{-\frac{n-2}{n-1}} \bar{Q}(t)\rt) \le 0. \]
\end{proof}
\section{A reverse Pensore inequality on asymptotically locally hyperbolic static manifolds}
Recall that we assume that $M$ has connected boundary. In this section, we further assume that $(M,g)$ is asymptotically hyperbolic.
\begin{defn}
$(M,g)$ is {\it conformally compact} if there exists a Riemannian manifold $(\bar{M}, \bar{g})$ such that
\begin{enumerate}
\item $\bar{M} = M \cup \pl_\infty M$,
\item there exists a defining function $\rho \in C^\infty(\bar{M})$ for $\pl_\infty M$. Namely, $\rho^{-1}(0) = \pl_\infty M$ and $d\rho \neq 0$ on $\pl_\infty M$.
\item $\bar{g} = \rho^2 g$ is a smooth metric on $\bar{M}.$
\end{enumerate} 
\end{defn}
It is well-known that $\bar{g}$ induces a well-defined conformal class on $\pl_\infty M$. 
\begin{defn}
A static manifold $(M,g,f)$ is {\it asymptotically locally hyperbolic} if
\begin{enumerate}
\item $(M,g)$ is conformally compact,
\item There is an Einstein metric $g_0$ in the conformal class $[\bar{g}|_{\pl_\infty M}]$. We normalize $g_0$ such that $Ric(g_0) = (n-2) k g_0$ with $k \in \{ -1,0,1\}$
\item $f^{-1}$ is a defining function and $f^{-2} g|_{\pl_\infty M} = g_0$.
\end{enumerate}
We say $(M,g,f)$ is {\it asymptotically hyperbolic} if the conformal boundary $(\pl_\infty M, [\bar{g}|_{\pl_\infty M}])$ is conformal to $(S^{n-1}, [g_0])$ where $g_0$ is the standard metric on $S^{n-1}$.  
\end{defn}
We remark that if $(M,g)$ is conformally compact, the sectional curvature approaches $-1$ at infinity. Therefore, an asymptotically hyperbolic static manifold must have negative cosmological constant.

\begin{Ex}\cite[Definition 1.2]{LN}
Let $(N,g_0)$ be an $(n-1)$-dimensional Einstein manifold  with $Ric(g_0) = (n-2)k g_0$. Let $m \in \R$ be large enough so that the function
\[ s^2 + k - 2m s^{2-n} =: V(r)^2 \] has a positive zero. Let $s_m $ be the largest zero of $V^2$, and define the metric
\[ \frac{1}{k + s^2 - 2m s^{2-n}} ds^2 + s^2 g_0 \]
defined on $(s_m, \infty) \times N$. Let $(M,g)$ be the metric completion of this Riemannian manifold. $(M,g,V)$ is a 1-parameter family of asymptotically locally hyperbolic static manifolds. We say that $(M,g,V)$ is a {\it Kottler space with conformal infinity $(N,g_0)$ and mass $m$.}

Let $\frac{\rho}{2} = \frac{1}{\sqrt{s^2 + k} + s}$. It is easy to see that the Kottler space with zero mass is conformally flat
\[ g = \rho^{-2} \lt( d\rho^2 + \lt( 1 - \frac{k\rho^2}{4} \rt)^2 g_0 \rt), V = \frac{1}{\rho} + \frac{k\rho}{4}. \]
\end{Ex} 

One implication of (\ref{Heintze-Karcher}) is a reverse Penrose inequality on asymptotically locally hyperbolic static manifold.
\begin{cor}[Corollary 1.3]
Let $(M,g,f)$ be an asymptotically locally hyperbolic static manifold with a connected boundary $N$. Assume (\ref{assumption, R^N}) holds on $N$. Then we have an upper bound for the mass of $(M,g)$:
\begin{align}
m \le \frac{\kappa}{(n-2)\,\omega_{n-1}} \lt( 1 - \frac{n-1}{\max_N \lt( \frac{R^N + n(n-1)}{2} \rt)}\rt) \, \mbox{vol}(N), 
\end{align}
where $\omega_{n-1}$ is the volume of the $(n-1)$-dimensional unit sphere.
\end{cor}

\begin{proof}
We identify the manifold with $(0,\epsilon) \times S^{n-1}$ near the conformal boundary. There exists a defining function $\rho$ such that the metric is given by
\[ g = \rho^{-2} \lt( d\rho^2 + h_\rho \rt). \]
By analyzing the static equations (see \cite[page 919]{W}), we have the expansion of the static potential and $h_\rho$ near the conformal boundary:
\begin{align}
f &= \frac{1}{\rho} + \frac{k\rho}{4} + \frac{\alpha}{2} \rho^{n-1} + o(\rho^{n-1}), \label{expansion f}\\
h_\rho &= \lt( 1 - \frac{k\rho^2}{4}\rt)^2 g_0 + \frac{\tau}{n} \rho^n + o(\rho^{n})\label{expansion h_rho},
\end{align}
where $\tau$ is a symmetric two-tensor on $S^{n-1}$ and $\alpha = - \frac{\mxtr_{g_0}{ \tau}}{n}$. Note that our $\tau$ differs from equation (17) in \cite{W} by a factor of $n$.

When $\rho$ is sufficiently small, the level set $\Sigma_\rho$ would have positive mean curvature with respect to $-\frac{\pl}{\pl \rho}$. Applying (\ref{Heintze-Karcher}) to such $\Sigma_\rho$ implies that
\[ (n-1) \int_{\Sigma_\rho} \frac{f}{H} d\mu - \int_{\Sigma_\rho} \frac{\pl f}{\pl \nu} d\mu \ge \kappa \lt( -1 + \frac{n-1}{\max_N \lt( \frac{R^N + n(n-1) }{2} \rt)} \rt) \mbox{vol}(N).\]
Here we have used the equation $\Delta_g f = nf$. 
Let $\gamma = d\rho^2 + h_\rho$. We compute the mean curvature of $\Sigma_\rho$ with respect to the outward normal $-\frac{\pl}{\pl\rho}$. Let $H_g$ and $H_\gamma$ denote the mean curvature in the ambient metric $g$ and $\gamma$ respectively. We have
\[ H_g = \rho H_\gamma + (n-1) \]
and
\[ H_\gamma = -\frac{1}{2} \lt( h_\rho^{-1} \rt)^{ij}  
\frac{\pl (h_\rho)_{ij}}{\pl\rho} . \]
By (\ref{expansion h_rho}), we have
\[ \frac{n-1}{H_g} = \frac{1 - \frac{k\rho^2}{4}}{1 + \frac{k\rho^2}{4}} + \frac{\mxtr_{g_0} \tau}{2(n-1)} \rho^n + o(\rho^n). \]
On the other hand, $\frac{\pl f}{\pl \nu} = -\rho \frac{\pl f}{\pl \rho} + o(\rho^{n-1})$ on $\Sigma_\rho$. By (\ref{expansion f}), we obtain on $\Sigma_\rho$,
\begin{align*}
(n-1) \frac{f}{H} + \rho \frac{\pl f}{\pl \rho} = \frac{(2-n)}{2(n-1)} (\mxtr_{g_0} \tau) \rho^{n-1} + o(\rho^{n-1}).
\end{align*}
Recall that the mass of $(M,g)$ is defined as 
\begin{align}
 m = \frac{1}{2(n-1) \omega_{n-1}} \int_{S^{n-1}} \mxtr_{g_0} \tau \,d\mu_{g_0}. \label{mass def}
\end{align}
Letting $\rho \rw 0$, we get
\[ (2-n) \omega_{n-1} \cdot m \ge \kappa \lt( -1 + \frac{n-1}{\max_N \lt( \frac{R^N + n(n-1) }{2} \rt)} \rt) \mbox{vol}(N). \]
This completes the proof .
\end{proof}

\begin{rmk}
We compare our result with that of Chru\'{s}ciel-Simon. When $n=3$, under the assumption of Corollary \ref{reverse Penrose}, the inequality in Theorem 1.5 of \cite{CS} (see also \cite[Theorem 2.5]{LN}) reads 
\[ m \le m_0 \]
where $m_0$ is the mass of the reference Kottler space that has the same conformal infinity and surface gravity $\kappa$. Denote $s_{m_0}$ by $s_0$. We have $m_0 = \frac{1}{2}(k s_0 + s_0^3)$ and $\kappa = \frac{3}{2} s_0 + \frac{k}{2s_0}$. Moreover, the horizon of the reference has area $4\pi s_0^2$ and constant scalar curvature $R_0 = 2 ks_0^{-2}$. Simple algebra shows 
\[ m_0 = \frac{\kappa}{4\pi} \lt( 1 - \frac{2}{\frac{R_0+6}{2}}\rt) area(N_0). \]
\end{rmk}

We note that the above corollary implies a simple case of \cite[Theorem I.3]{CS}.
\begin{cor}
Let $(M,g,f)$ be an asymptotically locally hyperbolic static manifold with positive ADM mass. Then $\pl M \neq \phi$. 
\end{cor}

\section{Static manifolds with positive cosmological constant}
When $0 < m < \frac{2}{n-2}\lt( \frac{n-2}{n} \rt)^{n/2}$ , $f^2 = 1-s^2-ms^{2-n}$ has two positive roots $s_1 < s_2$. With $m$ in this range, de Sitter-Schwarzschild manifold with parameter $m$ is defined as a warped product on $M = [s_1,s_2] \times S^n-1$ , $g = \frac{1}{f^2} ds^2 + s^2 g_0$ with potential $f$. The surface gravity on $N_1 = \{ s=s_1\}$ and $N_2= \{ s=s_2\}$ are
\[ \kappa_1 =  -s_1 + \frac{(n-2)m}{2} s_1^{1-n}, \quad \kappa_2 = s_2 - \frac{(n-2)m}{2} s_2^{1-n} \]
respectively. Note the sign difference. Straightforward computation shows that
\[ \kappa_1 \lt( \frac{2(n-1)}{R^{N_1} - n(n-1)} + 1\rt) \cdot vol(N_1) = \kappa_2 \lt( \frac{2(n-1)}{n(n-1) - R^{N_2}} -1 \rt) \cdot vol(N_2) = \frac{(n-2)m}{2} \omega_{n-1}. \]

Using (\ref{Heintze-Karcher}), the above identity becomes an inequality on static manifolds with two horizons. 
\begin{cor}[Corollary 1.5] Let $(M,g,f)$ be a static manifold with positive cosmological constant. Suppose that $M$ is diffeomorphic to $N \times [s_1, s_2]$. Let $N_1 = N \times \{ s_1\}$ and $N_2 = N \times \{ s_2\}$ be the two connected components of $\pl M$. Assume that \[ R^{N_1} > n(n-1)
\quad\mbox{ and} \quad R^{N_2}< n(n-1). \] Then we have
\begin{align}
\kappa_2 \int_{N_2} \lt( \dfrac{2(n-1)}{n(n-1) - R^{N_2}} -1 \rt) d\mu  \ge \kappa_1 \lt( \frac{2(n-1)}{\max_{N_1} \lt( R^{N_1} - n(n-1) \rt)} +1 \rt) \cdot vol(N_1).
\end{align}
Here $\kappa_1$ and $\kappa_2$ are the surface gravity $|Df|_g$ on $N_1$ and $N_2$ respectively.
\end{cor}     
\begin{proof}
Let $r$ be the distance function from $N_2$ and $\Sigma_r$ be the level sets of $r$. Since $R^{N_2} < n(n-1)$, $\Sigma_r$ has positive mean curvature with respect to $-\frac{\pl}{\pl r}$ when $r$ is sufficiently small. Since $\Delta_g f = -nf$, the divergence theorem implies 
\[ n \int_M f = -\kappa_1 \cdot vol(N_1) - \kappa_2 \cdot vol(N_2)\] Recall that facts $f = \kappa_2 r + O(r^3)$ and $H = \frac{R^{N_2} - n(n-1)}{2} r + O(r^2)$. The assertion follows by applying (\ref{Heintze-Karcher}) to $\Sigma_r$ and letting $r \rw 0$.  
\end{proof}

\end{document}